\documentclass[12pt,leqno,fleqn]{amsart}
\usepackage{hyperref}
\usepackage{amsmath,amstext,amsthm,amssymb} 

\usepackage{float}
\usepackage{amsrefs}
\usepackage{ulem}
\hypersetup{
    colorlinks=true,       
    linkcolor=blue,          
    citecolor=magenta,        
    filecolor=magenta,      
    urlcolor=cyan,          
}

\usepackage{pgf,pgfarrows} 
\usepackage{tikz}
\usetikzlibrary[decorations.pathreplacing]

\allowdisplaybreaks
\swapnumbers

\allowdisplaybreaks
\swapnumbers

\theoremstyle{plain} 
\newtheorem{lemma}[equation]{Lemma} 
\newtheorem{proposition}[equation]{Proposition} 
\newtheorem{theorem}[equation]{Theorem} 
\newtheorem{corollary}[equation]{Corollary}

\theoremstyle{definition}
\newtheorem{definition}[equation]{Definition} 

\theoremstyle{remark}
\newtheorem{remark}[equation]{Remark}

\numberwithin{equation}{section}

\def\norm#1.#2.{\lVert#1\rVert_{#2}}
\def\Norm#1.#2.{\bigl\lVert#1\bigr\rVert_{#2}}
\def\NOrm#1.#2.{\Bigl\lVert#1\Bigr\rVert_{#2}}
\def\NORm#1.#2.{\biggl\lVert#1\biggr\rVert_{#2}}
\def\NORM#1.#2.{\Biggl\lVert#1\Biggr\rVert_{#2}}

\def\ip#1,#2,{\langle #1,#2\rangle}
\def\Ip#1,#2,{\langle#1,#2\rangle}
\def\IP#1,#2,{\langle#1,#2\rangle}

\def\mid{\,:\,}

\def\XXint#1#2#3{{\setbox0=\hbox{$#1{#2#3}{\int}$}
     \vcenter{\hbox{$#2#3$}}\kern-.5\wd0}}

\newcommand{\zat}{\mathbb{Z}}

\newcommand{\re}{\operatorname{Re}}
\newcommand{\im}{\operatorname{Im}}
\newcommand{\qal}{|Q_{I}|_{\alpha}}
\newcommand{\lpwal}{L^{p}(wdA_{\alpha})}
\begin{document}

\title[Sharp B{\'e}koll{\'e} estimates for the Bergman Projection]{Sharp B{\'e}koll{\'e} estimates for the Bergman Projection}


\author[S. Pott]{Sandra Pott}
\thanks{}
\address{Centre for Mathematical Sciences, University of Lund, Lund, Sweden}
\email {sandra@maths.lth.se}

\author[M.C. Reguera]{Maria Carmen Reguera}
\thanks{Supported by Lund University, Mathematics in the Faculty of Science with a postdoctoral research grant}
\address{Centre for Mathematical Sciences, University of Lund, Lund, Sweden}
\email {mreguera@maths.lth.se}


\begin{abstract}
We prove sharp estimates for the Bergman projection in weighted Bergman spaces in terms of the B{\'e}koll{\'e} constant. Our main tools are a dyadic model dominating the operator and an adaptation of a method of Cruz-Uribe, Martell and P\'erez.
\end{abstract}

\maketitle

\section{Introduction}

For $1<p<\infty$ and $-1<\alpha< \infty$, we consider the weighted Bergman space $A^{p}_{\alpha}(\mathcal H)$, defined as the space of analytic functions in $L^{p}( dA_{\alpha})$, where $ dA_{\alpha}(z)=(\im z)^{\alpha}dA(z)$ and $dA$ stands for the area measure in the upper half plane $\mathcal H=\{z\in \mathbb C\mid \, \im z>0\}$.

We consider the natural projection onto these spaces, i.e., the projection from $L^{2}( dA_{\alpha})$ onto $A^{2}_{\alpha}(\mathcal H)$, also known as the Bergman projection. It is easy to obtain the following integral representation for the Bergman projection,
\begin{equation}\label{e.pbp}
P_{\alpha}f(z)=c_{\alpha}\int_{\mathcal H} \frac{f(\xi)}{(z-\bar{\xi})^{2+\alpha}} dA_{\alpha}(\xi),
\end{equation}
where $c_{\alpha}$ is an appropriate real constant.
A classical result states that $P_{\alpha}\mid  L^{p}( dA_{\alpha})\mapsto  A^{p}_{\alpha}(\mathcal H)$. Moreover if we consider the maximal Bergman projection
\begin{equation}\label{e.pbpplus}
P_{\alpha}^{+}f(z)=\int_{\mathcal H} \frac{|f|(\xi)}{|z-\bar{\xi}|^{2+\alpha}} dA_{\alpha}(\xi),
\end{equation}
then $P_{\alpha}^{+} \mid  L^{p}( dA_{\alpha})\mapsto  L^{p}( dA_{\alpha})$ as well, thus indicating that the role of cancellation in this family of operators is very small. For these classical results see e.g. \cite{MR1758653}.

The boundedness of the Bergman projection in weighted $L^{p}$ spaces was studied by B{\'e}koll{\'e} and Bonami in \cite{MR497663} and \cite{MR667319}. The
 main result is stated in the following

\begin{theorem}[B{\'e}koll{\'e}-Bonami]
\label{t.bandb}
Let $1<p<\infty$ and $-1<\alpha< \infty$, let $w$ be a positive locally integrable function. The following assertions are equivalent:
\begin{enumerate}
\item 
$$
P_{\alpha}\mid  L^{p}( wdA_{\alpha})\mapsto L^{p}( wdA_{\alpha})
$$
\item 
$$
P_{\alpha}^{+}\mid  L^{p}( wdA_{\alpha})\mapsto L^{p}( wdA_{\alpha})
$$
\item 
$$
w\in B_{p, \alpha}
$$
\end{enumerate}
where the $B_{p, \alpha}$ class of weights is defined as the family of positive locally integrable functions for which
\begin{equation}
\label{e.bekolle}
B_{p, \alpha}(w):=\sup_{I\subset \mathbb R, \text{ interval}}\frac{|Q_{I}|_{w,\alpha}}{|Q_{I}|_{\alpha}} \left(\frac{|Q_{I}|_{w^{1-p'}, \alpha}}{|Q_{I}|_{\alpha}}\right)^{p-1}<\infty,
\end{equation}
where $Q_{I}$ is the Carleson box associated to interval $I$ (see \eqref{e.cbox}), $|Q_{I}|_{\alpha}:=\int_{Q_{I}} dA_{\alpha}\approx|Q_{I}|^{1+\alpha/2}$ and $|Q_{I}|_{w,\alpha}=\int_{Q_{I}} wdA_{\alpha}$.

\end{theorem}

Our aim in this paper is to study the dependence of the operator norm,\\  $\displaystyle \norm P_{\alpha}.L^{p}(wdA_{\alpha})\mapsto L^{p}(wdA_{\alpha}).$ on the $B_{p, \alpha}(w)$ constant. We first find estimates for the maximal operator $P_{\alpha}^{+}$, which allows us to deduce the same estimates for $P_{\alpha}$. Our main theorem is the following:

\begin{theorem}
\label{t.bekolle}
For $1<p<\infty$, let $w\in B_{p, \alpha}$ be a B{\'e}koll{\'e} weight with constant $B_{p, \alpha}(w)$ and let $P_{\alpha}^{+}$ be the maximal Bergman projection. Then
\begin{equation}
\label{e.sharpbeko}
\norm P_{\alpha}^{+}f. L^{p}(wdA_{\alpha}).\leq C B_{p, \alpha}(w)^{\max(1, \frac{1}{p-1})} \norm f. L^{p}(wdA_{\alpha}).,
\end{equation}
where the constant $C$ depends only on $p$ and $\alpha$.
\end{theorem}

\begin{corollary}\label{c.sharpb}
For $1<p<\infty$, let $w\in B_{p, \alpha}$ be a B{\'e}koll{\'e} weight with constant $B_{p, \alpha}(w)$ and let $P_{\alpha}$ be the Bergman projection. Then
\begin{equation}
\label{e.sharpbekoc} 
\norm P_{\alpha}f. L^{p}(wdA_{\alpha}).\leq C B_{p}(w)^{\max(1, \frac{1}{p-1})} \norm f. L^{p}(wdA_{\alpha})..
\end{equation}
\end{corollary}

We remark that the estimate is sharp. In Section 5 we will show examples of weights and functions where the sharp bound is attained for $P_{\alpha}$.

The question is motivated by the recent developments on the $A_{2}$-Conjecture for singular integrals in the setting of Muckenhoupt weighted $L^{p}$ spaces. The search for sharp estimates in terms of the Muckenhoupt $A_p(w)$ constant  (defined analogously to the $B_{p, \alpha}(w)$ with the difference that the supremum over all cubes, rather than just Carleson squares, is taken) has received a lot of attention lately. Sharp estimates for the norm of general  Calder\'on-Zygmund operators in terms of the $A_p(w)$ constant were ultimately found by Hyt{\"o}nen \cite{1007.4330}. For a full account of the history and the contributors to the now $A_2$-Theorem see e.g. \cites{MR1894362, 1007.4330, 1202.2824,1202.2229} and the references (of arbitrary order) therein.
 
In the Bergman setting very little has been known about the dependence of the operator norm on the B{\'e}koll{\'e} constant. The work of Alexandru and Constantin \cite{Aleman20122359} characterize the boundedness of the Bergman projection on weighted vector-valued spaces in terms of a vector-valued B{\'e}koll{\'e}-type condition. In \cite{Aleman20122359}, the authors managed to control the norm of the Bergman projection by the B{\'e}koll{\'e} constant associated to the weight $w$ raised to the power $5/2$. One has to take into account that the vector-valued case presents extra difficulties and at this point we do not know how to extent our methods to this more delicate setting.

Our proof strategy follows the ideas already used in the singular integral case. We focus on the case $p=2$, as the general case follows by an adaptation of Rubio de Francia's extrapolating method to our setting. First we construct a dyadic model that dominates the Bergman projection. Here, one should note that the already obtained generalizations of the $A_{2}$-Theorem to metric  spaces \cite{1106.1342} do not cover this setting, since the $B_{p,\alpha}$ condition is only tested on Carleson squares. So the approach of dyadic models appears to be novel in this setting. On the other hand, our setting is much less delicate than that of the $A_2$-Theorem, since cancellation plays no role here. As in the case of the $A_{2}$-Theorem such sharp estimates can be obtained from two weight estimates, see \cite{MR2657437}. For this and connections with the Sarason Conjecture, see \cite{APR}. Here, we use the easier approach from Cruz-Uribe, Martell and P\'erez in \cites{MR2628851,MR2854179}.
As in the latest work of Lerner \cite{1202.2824} and Hyt{\"o}nen, Lacey and P\'erez \cite{1202.2229}, we only need finitely many dyadic grids. 

The paper is organized as follows: Section 2 includes some of the basic concepts and the results we are going to use throughout the paper. In Section 3 we present the dyadic operator and the Proposition \ref{p.pointwise}, that allows to control the Bergman projection by it. Section 4 is dedicated to the proof of the Main Theorem. Section 5 provides the reader with the examples that confirm sharpness of our estimates. The last section is dedicated to the bibliography.

\section{Basic Concepts}

We start this section by introducing the notation that will be relevant to us throughout these notes.
The space we will be working on is the upper half of the complex plane, $\mathcal H$. For an interval $I\subset \mathbb R$, we define the Carleson cube associated to $I$, which we denote by $Q_{I}$, and the top half associated to $I$, denoted by $T_{I}$, as

\begin{equation}
\label{e.cbox}
Q_{I}= I \times [0, |I|]; \quad T_{I}= I \times (|I|/2, |I|].
\end{equation}

$\mathcal D$ denotes the usual dyadic grid in $\mathbb R$, namely 
$\mathcal D = \{ \left.\left [ 2^{j}m, 2^{j}(m+1)\right.\right );\,\,\,m, j \in \mathbb Z \}$. In general, a dyadic grid in $\mathbb R$ is a collection of intervals such that for each interval $I$ we have the following,
\begin{enumerate}
\item The set $\{J\in \mathcal D \mid |J|=|I| \}$ forms a partition of $\mathbb R$.

\item The interval $I$ is the union of two intervals $I_{+}$ and $I_{-}$, and $|I_{\pm}|=\frac{1}{2} |I|$
\end{enumerate}

\begin{remark}
\label{r.tiling}
Notice that the family $\{T_{I}\}_{I\in \mathcal D}$, where $\mathcal D$ is a dyadic grid in $\mathbb R$, provides a tiling of the half plane $\mathcal H$.
\end{remark}

A weight $w$ will be a non-negative locally integrable function on the half plane $\mathcal H$. Let $M_{w,\alpha}$ be the dyadic maximal function associated to Carleson cubes on $\mathcal H$ with respect to measure $wdA_{\alpha}$, i.e, for $f\in L^{1}_{loc}(\mathcal H)$ and a dyadic grid $\mathcal D$ of $\mathbb R$ we define

\begin{equation}
\label{e.max}
M_{\alpha,w}f(z)= \sup_{I \in \mathcal D } \frac{1_{Q_{I}}}{|Q_{I}|_{w,\alpha}}\int_{Q_{I}} |f|wdA_{\alpha},
\end{equation}

where $1_{E}$ denotes the characteristic function of the set $E\subset \mathcal H $ and $|Q_{I}|_{w,\alpha}=\int_{Q_{I}} wdA_{\alpha}$.

The following classical result asserts the boundedness properties of the maximal operator defined above \eqref{e.max}.

\begin{theorem}
\label{t.mbounded}
Let $1<p<\infty$ and let $w$ be a weight. Then
\begin{equation}
\label{e.mbounded}
\norm M_{\alpha, w}f. L^{p}(wdA_{\alpha}).\leq C \norm f.L^{p}(wdA_{\alpha}).
\end{equation}
where the constant $C$ is independent of the weight $w$.
\end{theorem}
For our purposes it is important that the constant $C$ does not depend on the weight $w$. Note that this follows directly from the fact that the Hardy-Littlewood maximal function is bounded on $L^{p}(\mathbb R^{n}, \mu )$ for any $1<p<\infty$ and $n\in \mathbb N$, with a constant only depending on $p$ and $n$. For a proof of this classical result see, for instance, \cite{MR767633}.

In the sequel, the numbers $p,p'$ will satisfy $1<p,p'<\infty$ and $\frac{1}{p}+\frac{1}{p'}=1$. 
Throughout the paper and abusing the notation, we will use the same symbol $w$ for a measure and its Lebesgue density.

The concept of $B_{2,\alpha}$ weight can be extended to any pair of weights $w$ and $\sigma$. We define it below, as we will make use of it in Section 4 when describing the extrapolation algorithm (see \eqref{e.pasoint}).

\begin{definition}
\label{d.jointb2}
Let $w$ and $\sigma$ be two weights. We say the weights $(w,\sigma)$ belong to the joint $B_{2,\alpha}$ class with associated constant $B_{2,\alpha}(w,\sigma)$ if
\begin{equation}
\label{e.jointbekolle}
B_{2, \alpha}(w,\sigma):=\sup_{I\subset \mathbb R, \text{ interval}}\frac{|Q_{I}|_{w,\alpha}}{|Q_{I}|_{\alpha}} \frac{|Q_{I}|_{\sigma^{-1}, \alpha}}{|Q_{I}|_{\alpha}}<\infty.
\end{equation}
\end{definition}

\section{A dyadic model for the maximal Bergman projection}

In this section we will use shifts of the usual dyadic grid in $\mathbb R$ to dominate the maximal Bergman projection by a finite family of discrete dyadic operators. The idea of considering shifts of the original dyadic grid to control non-dyadic objects (space norms, operators) by dyadic ones is well known to Harmonic Analysts and goes back to Garnett and Jones \cite{MR658065} and Christ \cite{MR951506} independently, we also refer the reader to the work of Tao Mei \cite{MR1993970}. In very recent work on the improvement of the $A_2$ Theorem (\cite{1202.2824}, \cite{1202.2229}), the authors avoid the use of Hyt\"onen's original probabilistic arguments and random dyadic grids and consider instead a finite family of shifted dyadic grids.  In the same spirit, we aim to obtain similar results for the maximal Bergman projection. 

Let us consider the following system of dyadic grids \\
$\mathcal D^{\beta}:= \left\{ 2^{j}([0,1)+ m + (-1)^{j}\beta): m\in \mathbb Z, j\in \mathbb Z  \right\}$ for $\beta\in \{0,1/3\}$. These systems have been recently considered by Hyt{\"o}nen and P\'erez \cite{1103.5562}, Lerner \cite{1202.2824} and Lacey, Hyt{\"o}nen and P\'erez \cite{1202.2229} in their simplified versions of the $A_2$ Theorem. The following lemma is a well-known fact, we include the proof here for the sake of completeness.

\begin{lemma}\label{l.grids}
Let $I$ be any interval in $\mathbb R$. Then there exists an interval $K\in \mathcal D^{\beta}$ for some $\beta\in \{0,1/3\}$ such that $I\subset K$ and $|K|\leq 8|I|$.
\end{lemma}

\begin{proof}
For any interval $I$, we know there exists an integer $l$ such that $2^{l}\leq |I|\leq 2^{l+1}$. Suppose $I$ does not contain any point of the form $2^{l+3}m$, then there exists $K\in \mathcal D^{0}$, $K:= [ 2^{l+3}m_{0}, 2^{l+3}(m_{0}+1))$ for some $m_{0}\in \mathbb Z$ such that $I\subset K$ and $|K|=2^{l+3}\leq 8|I|$. 
Suppose on the contrary that $I$ contains a point of the form $2^{l+3}m_{1}$ for some integer $m_{1}$. We write the case when $l+3$ is even, the odd case is identical. Then $I$ cannot contain any point of the form $2^{l+3}(m+1/3)$ as if that was the case,
$2^{l+3}|m_{1}- (m+1/3)|\leq |I|$ and this inequality contradicts the fact that $|I|\leq 2^{l+1}$. So once again there exists $K\in \mathcal D^{1/3}$ such that $I\subset K$ and $|K|\leq 8|I|$.
\end{proof}

We are now ready to define the dyadic model operator that will be relevant to us, see also \cite{APR}.

\begin{definition}\label{d.dyadicmodel}
Let $\mathcal D^{\beta}$ be one of the dyadic grids in $\mathbb R$ described above. We define the positive dyadic operator
\begin{equation}\label{e.dyadicp}
Q_{\alpha}^{\beta} f=\sum_{I\in \mathcal D^{\beta}} \langle f, \frac{1_{Q_{I}}}{|I|^{2+\alpha}} \rangle_{\alpha} 1_{Q_{I}},
\end{equation}
where $\langle \cdot, \cdot \rangle_{\alpha}$ stands for the standard scalar product in $\mathcal H$ with respect to measure $dA_{\alpha}$.
\end{definition}

The following proposition proves the relation between the maximal Bergman projection and the dyadic operators described above.

\begin{proposition}\label{p.pointwise}
There exists $C$ such that for all $f\in L^{1}_{\textup{loc}}$, $f\geq 0$ and $z\in \mathcal H$
\begin{equation}
\label{e.pointwise}
P^{+}_{\alpha}f(z)\leq C \mathbb \sum_{\beta\in \{0,1/3\}} Q_{\alpha}^{\beta} f(z).
\end{equation}
\end{proposition}

\begin{proof}
Let $K_{\alpha}^{+}$ be the kernel associated to the maximal Bergman projection and $K_{\beta, \alpha}$ be the kernel associated to $Q_{\alpha}^{\beta}$, i.e., for $z,\xi\in \mathcal H$

$$
K_{\alpha}^{+}(z,\xi)=\frac{1}{|z-\bar{\xi}|^{2+\alpha}} \,\, \text{and} \quad K_{\alpha}^{\beta}(z,\xi)=\sum_{I\in\mathcal D_{\beta}}\frac{1_{Q_{I}}(z)1_{Q_{I}}(\xi)}{|I|^{2+\alpha}}.
$$

Therefore it is enough to prove that for every $z,\xi\in \mathcal H$,

\begin{equation}
\label{e.pointwisek}
K_{\alpha}^{+}(z,\xi)\leq c_{2} \sum_{\beta\in \{0,1/3\}} K_{\alpha}^{\beta}(z,\xi),
\end{equation}
with $c_{2}$ independent of $z$ and $\xi$.
Let us fix $z,\xi \in \mathcal H$. Then there exists $l\in \zat$ such that
\begin{equation}
\label{e.pigeon}
2^{2l}\leq (\re z-\re \xi)^{2}+(\im z+\im \xi)^{2} \leq 2^{2l+2}.
\end{equation}

It is easy to see that there exists an interval $I$ (not necessarily dyadic) such that $2^{l-1}<|I|\leq 2^{l+1}$ and $z, \xi \in Q_{I}$. By Lemma $\ref{l.grids}$, we find $K\in \mathcal D^{\beta}$ for some $\beta\in \{0, 1/3\}$ such that $I\subset K$ and $|K|\leq 8|I|$. Now the proof of the theorem follows from the set of inequalities below:

\begin{eqnarray*}
\frac{1}{|z-\bar{\xi}|^{2+\alpha}}& \leq & \frac{1}{2^{l(2+\alpha)}}\\
 & \leq & 2^{(2+\alpha)} 8^{2+\alpha} \frac{1}{|K|^{2+\alpha}} \\
 & \leq & C_{\alpha}\sum_{\substack{I\in  \mathcal D^{\beta}\\ K\subset I}}\frac{1_{Q_{I}}(z)1_{Q_{I}}(\xi)}{|I|^{2+\alpha}}\\
 & \leq & C_{\alpha}  \sum_{\beta\in \{0,1/3\}} K_{\alpha}^{\beta}(z,\xi).
\end{eqnarray*}

\end{proof}

\section{Proof of the Main theorem}

The proof of the main theorem is a consequence of the corresponding result for the dyadic operators $Q_{\alpha}^{\beta}$ thanks to Proposition \ref{p.pointwise}.
\begin{theorem}\label{t.sharpdyad}
Let $w\in B_{2,\alpha}$ with associated constant $B_{2,\alpha}(w)$. Let $Q_{\alpha}^{\beta}$ 
be the dyadic operator described in Definition \ref{d.dyadicmodel} for some $\beta\in \{0,1/3\}$ and $-1<\alpha <\infty$. Then
$$
\norm Q_{\alpha}^{\beta} f.L^{2}(wdA_{\alpha}). \leq C_{\alpha}B_{2,\alpha}(w)\norm f.L^{2}(wdA_{\alpha}).   .
$$
\end{theorem}

\begin{proof}
We claim that the following assertions are equivalent. The proof is an easy exercise that we are not going to include.
\begin{enumerate}
\item 
$$
Q_{\alpha}^{\beta} \mid L^{2}(wdA_{\alpha}) \mapsto L^{2}(wdA_{\alpha})
$$
\item 
$$
Q_{\alpha}^{\beta}(w^{-1}\cdot) \mid L^{2}(w^{-1}dA_{\alpha}) \mapsto L^{2}(wdA_{\alpha})
$$
\end{enumerate}
Moreover the norms of the two operators are equal.
With this in mind, we need to consider $\norm Q_{\alpha}^{\beta}(w^{-1}f). L^{2}(wdA_{\alpha}).$, for $f\in  L^{2}(w^{-1}dA_{\alpha})$. We use duality,
\begin{equation}
\label{e.duality}
\norm Q_{\alpha}^{\beta}(w^{-1}f). L^{2}(wdA_{\alpha}).=\sup_{\substack{0\leq g\in L^{2}(wdA_{\alpha})\\ \norm g.L^{2}(wdA_{\alpha}).=1 }}\int_{\mathcal H} Q_{\alpha}^{\beta}(w^{-1}f)gwdA_{\alpha}
\end{equation}

So let us study the right hand side of \eqref{e.duality}. Let  $0\leq f\in  L^{2}(w^{-1}dA_{\alpha})$ and $g\in L^{2}(wdA_{\alpha})$ as above. Then

\begin{eqnarray*}
& & \int_{\mathcal H} Q_{\alpha}^{\beta}(w^{-1}f)gwdA_{\alpha}\\
 &=& \sum_{I\in \mathcal D^{\beta}}\ip w^{-1}f, 1_{Q_{I}},_{\alpha} \ip wg, 1_{Q_{I}},_{\alpha}|Q_{I}|^{-1-\alpha/2}\\
 & \approx & \sum_{I\in \mathcal D^{\beta}}|Q_{I}|_{\alpha}\left( \frac{1}{|Q_{I}|_{w^{-1},\alpha}}\int_{Q_{I}} f w^{-1}dA_{\alpha}\right)\left(  \frac{1}{|Q_{I}|_{w,\alpha}}\int_{Q_{I}} g wdA_{\alpha} \right)  \frac{|Q_{I}|_{w^{-1},\alpha}}{|Q_{I}|_{\alpha}}  \frac{|Q_{I}|_{w,\alpha} }{|Q_{I}|_{\alpha}} \\
 & \leq & C_{\alpha} B_{2, \alpha}(w) \sum_{I\in \mathcal D^{\beta}}|T_{I}|_{\alpha} \left(\frac{1}{|Q_{I}|_{w^{-1},\alpha}}\int_{Q_{I}} f w^{-1}dA_{\alpha} \right) \left( \frac{1}{|Q_{I}|_{w,\alpha}}\int_{Q_{I}} g wdA_{\alpha}\right)\\
 & \leq  & C_{\alpha} B_{2, \alpha}(w) \sum_{I\in \mathcal D^{\beta}}\int_{T_{I}} \left(\frac{1}{|Q_{I}|_{w^{-1},\alpha}}\int_{Q_{I}} f w^{-1}dA_{\alpha}\right)  \left(\frac{1}{|Q_{I}|_{w,\alpha}}\int_{Q_{I}} g wdA_{\alpha}\right) dA_{\alpha}\\
 & \leq  & C_{\alpha} B_{2, \alpha}(w) \int_{\mathcal H} \left(M_{w^{-1},\alpha}f\right) \left (M_{w,\alpha}g\right) w^{-1/2} w^{1/2}dA_{\alpha}\\
 & \leq  & C_{\alpha} B_{2, \alpha}(w) \norm f. L^{2}(w^{-1}dA_{\alpha}). \norm g.L^{2}(wdA_{\alpha}).,\\
\end{eqnarray*}

where for the last inequality we have used the estimates in \eqref{e.mbounded}. In the previous one we have used Remark \ref{r.tiling} and H{\"o}lder's inequality with respect to measure $dA_{\alpha}$. The constant $C_{\alpha}$ is not necessarily the same one in each line of the previous inequalities, but it is ultimately a constant that only depends on $\alpha$.

\end{proof}

Notice that we have only proven estimates for $p=2$. The case $p\neq 2$ follows from the case $p=2$, i.e, Theorem \ref{t.sharpdyad}, and an adaptation of Rubio de Francia's extrapolation argument, see Proposition \ref{p.extrap} below. Rubio de Francia's extrapolation theorem provides a very powerful tool that reduces the weighted $L^p$ boundedness for weights in the Muckenhoupt $A_p$ class to the study of one special exponent $p_0$ (typically $p_0=2$). For further details on the theory of extrapolation, we refer the reader to Garc\'ia-Cuerva and Rubio de Francia's book \cite{MR807149} or the new book by Cruz-Uribe, Martell and P\'erez \cite{MR2797562}. 
A sharp version of Rubio de Francia's extrapolation theorem was obtained by Dragi{\v{c}}evi{\'c}, Grafakos, Pereyra, and Petermichl \cite{MR2140200}. Below is an adaptation of their proof to our setting.

Let us consider the case $p>2$. The case $1<p<2$ follows from the following remark: 
\begin{remark}
\label{r.consid}
Let us $p$ and $p'$ be dual exponents. Then the use of duality allow us to conclude the following facts:
\begin{enumerate}
\item
$ 
\norm P_{\alpha}.L^{p}(wdA^{\alpha})\mapsto L^{p}(wdA^{\alpha}).= \norm P_{\alpha}.L^{p'}(w^{1-p'}dA^{\alpha})\mapsto L^{p}(w^{1-p'}dA^{\alpha}).,
$
\item $B_{p, \alpha}(w)^{\frac{1}{p-1}}= B_{p',\alpha}(w^{1-p'})$.
\end{enumerate}
\end{remark}

We are now ready to state and prove the extrapolation theorem.
\begin{proposition}
\label{p.extrap}
Let $T$ be an operator such that 
\begin{equation}
\label{e.b2true}
\norm T.L^{2}(wdA_{\alpha})\mapsto  L^{2}(wdA_{\alpha}).\leq C B_{2,\alpha}(w) \quad \text{for all } w\in B_{2,\alpha},
\end{equation}
then 
\begin{equation}
\label{e.bptrue}
\norm T.L^{p}(wdA_{\alpha})\mapsto  L^{p}(wdA_{\alpha}).\leq \tilde{C} B_{p,\alpha}(w) \quad \text{for all } w\in B_{p,\alpha} \text{ and } p>2,
\end{equation}
where $C, \tilde{C}$ are constants independent of the weights $w$.
\end{proposition}

\begin{proof}
We first notice that given the maximal function $M_{\alpha}$ with respect to measure $dA_{\alpha}$, i.e., 
$$
M_{\alpha}f(z)= \sup_{I\subset \mathbb R \text{ interval }  } \frac{1_{Q_{I}}}{\qal}\int_{Q_{I}} |f|dA_{\alpha}
$$
and $w\in B_{p,\alpha}$, then 
$$
M_{\alpha}\mid \lpwal \mapsto \lpwal,
$$
moreover 
\begin{equation}
\label{e.sharpalmax}
\norm M_{\alpha}.\lpwal \mapsto \lpwal .\leq B_{p,\alpha}(w)^{\frac{1}{p-1}}.
\end{equation} 

In the case of Muckenhoupt $A_p$ weights, with the measure $dA_{\alpha}$ replaced by Lebesgue measure and the usual Hardy-Littlewood Maximal operator in $\mathbb R^{n}$, the result was proven by Buckley \cite{MR1124164}. There exists a very elementary proof of \eqref{e.sharpalmax} due to Lerner \cite{MR2399047} that can be easily updated to this setting. We leave the details to the interested reader.

Letting $p>2$, we find sharp estimates for $Tf$ using \eqref{e.b2true}. Note that
\begin{equation}\label{e.ext1}
\norm Tf.L^{p}(wdA_{\alpha}). ^{2}= \sup_{\substack{0\leq h \\ \norm h.L^{p'/\phi(p)}(wdA_{\alpha}).=1}}\int_{\mathcal H}|Tf|^{2} h wdA_{\alpha},
\end{equation}
where $\phi(p)=\frac{p-2}{p-1}$.

We define $S_{w,\alpha}(h)=\left(\frac{M_{\alpha}(|h|^{1/\phi(p)}w)}{w}\right)^{\phi(p)}$. Then we use \eqref{e.sharpalmax} to conclude that $S_{w,\alpha}\mid L^{p'/\phi(p)}(wdA_{\alpha}) \mapsto L^{p'/\phi(p)}(wdA_{\alpha})$, moreover we have
\begin{equation}
\label{e.norms}
\norm S_{w,\alpha}. L^{p'/\phi(p)}(wdA_{\alpha}) \mapsto L^{p'/\phi(p)}(wdA_{\alpha}). \leq B_{p,\alpha}(w)^{\phi(p)}.
\end{equation}

We define the operator 
$$
D(h)=\sum_{k=0}^{\infty}\frac{1}{2^{k}}\frac{S_{w,\alpha}^{k}(h)}{\norm S_{w,\alpha}.L^{p'/\phi(p)}(wdA_{\alpha}).^{k}}
$$

We have the following properties associated to $D(h)$,
\begin{enumerate}
\item $|h|\leq D(h)$
\item $\norm D(h).L^{p'/\phi(p)}(wdA_{\alpha}).\leq 2 \norm h.L^{p'/\phi(p)}(wdA_{\alpha}).$
\item $S_{w,\alpha}(D(h))\leq 2 \norm S_{w,\alpha}.L^{p'/\phi(p)}(wdA_{\alpha}). D(h)$ and moreover
\begin{equation}
\label{e.b2bp}
B_{2,\alpha}(D(h)w)\leq c B_{p,\alpha}(w).
\end{equation}
\end{enumerate}
We only need to verify \eqref{e.b2bp}, as the other properties follow immediately by the definition of $D(h)$ and the subadditivity of $S_{w,\alpha}$. We first claim that for any $h\geq 0$,
\begin{equation}
\label{e.pasoint}
B_{2,\alpha}(hw, S_{w,\alpha}(h)w)\leq B_{p,\alpha}(w)^{1-\phi(p)}=B_{p,\alpha}(w)^{\frac{1}{p-1}},
\end{equation}
where we remind the reader that the joint $B_{2,\alpha}$ constant appearing on the left hand side of \eqref{e.pasoint} was previously defined in Definition \ref{d.jointb2}.
The proof of the claim goes as follows. Let $I$ be any interval in $\mathbb R$ and let $Q_{I}$ be the Carleson box associated to it. We call 
$$
B_{2,\alpha}(hw, S_{w,\alpha}(h)w)(I):=\left(\frac{1}{\qal}\int_{Q_{I}}hwdA_{\alpha}\right) \left(\frac{1}{\qal}\int_{Q_{I}}( S_{w,\alpha}(h)w)^{-1}  dA_{\alpha}\right).
$$
Then 
\begin{eqnarray*}
B_{2,\alpha}(hw, S_{w,\alpha}(h)w)(I) & = & 
\left(\frac{1}{\qal}\int_{Q_{I}}hwdA_{\alpha}\right) \left(\frac{1}{\qal}\int_{Q_{I}}( M_{\alpha}(h^{1/\phi(p)}w))^{-\phi(p)}w^{1-p'} dA_{\alpha}\right)\\
&\leq & \left(\frac{1}{\qal}\int_{Q_{I}}h^{1/\phi(p)}wdA_{\alpha}\right)^{\phi(p)} \left(\frac{1}{\qal}\int_{Q_{I}}wdA_{\alpha}\right)^{\frac{1}{p-1}}\\
&       & \times  \left(\frac{1}{\qal}\int_{Q_{I}}h^{1/\phi(p)}wdA_{\alpha}\right)^{-\phi(p)} \left(\frac{1}{\qal}\int_{Q_{I}}w^{1-p'} dA_{\alpha}\right)\\
&\leq & B_{p,\alpha}(w)^{\frac{1}{p-1}},
\end{eqnarray*}
where we have used that $1-\phi(p)=p'-1$, H{\"o}lder's inequality and the fact that for $z\in Q_{I}$, 
$$
M_{\alpha}(h^{1/\phi(p)}w)(z)\geq \frac{1}{\qal}\int_{Q_{I}}h^{1/\phi(p)}wdA_{\alpha}.
$$

We now use (3) and estimate \eqref{e.pasoint} to prove \eqref{e.b2bp},

\begin{eqnarray*}
B_{2,\alpha}(D(h)w)&= &\sup_{I\subset \mathbb R, \text{ interval}}\left(\frac{1}{|Q_{I}|_{\alpha}}\int_{Q_{I}}D(h)wdA_{\alpha}\right) \left(\frac{1}{|Q_{I}|_{\alpha}}\int_{Q_{I}}(D(h)w)^{-1}dA_{\alpha}\right)\\
&\leq & 2 \norm S_{w,\alpha}.L^{p'/\phi(p)}(wdA_{\alpha}). B_{2,\alpha}(D(h)w,S_{w,\alpha}(D(h)w)\\
 & \leq & 2 \norm S_{w,\alpha}.L^{p'/\phi(p)}(wdA_{\alpha}). B_{p,\alpha}(w)^{\frac{1}{p-1}}\\
 &\leq & CB_{p,\alpha}(w).
\end{eqnarray*}

We conclude the proof of the proposition using the hypothesis \eqref{e.b2bp}. Using \eqref{e.ext1}, it is enough to consider the estimate below for $h\geq 0,\, h\in L^{p'/\phi(p)}(wdA_{\alpha})$, $\norm h.L^{p'/\phi(p)}(wdA_{\alpha}).=1$,

\begin{eqnarray*}
\int_{\mathcal H}|Tf|^{2} h wdA_{\alpha}&\leq & \int_{\mathcal H}|Tf|^{2} D(h) wdA_{\alpha}\\
 &\leq & C B_{2,\alpha}(D(h)w)^{2}\int_{\mathcal H}|f|^{2}D(h) wdA_{\alpha}\\
 &\leq & C B_{p,\alpha}(w)^{2}\norm f.L^{p}(wdA_{\alpha}).^{2} \norm D(h).L^{(p/2)'}(wdA_{\alpha}).\\
 &\leq & C B_{p,\alpha}(w)^{2}\norm f.L^{p}(wdA_{\alpha}).^{2} \norm D(h).L^{p'/\phi(p)}(wdA_{\alpha}).\\
 &\leq & C B_{p,\alpha}(w)^{2}\norm f.L^{p}(wdA_{\alpha}).^{2}.
\end{eqnarray*}

\end{proof}

\section{Sharp examples}

It is enough to consider the case $1<p\leq 2$. The other case follows from duality considerations, see Remark \ref{r.consid}.

Let $1<p\leq 2$, $-1<\alpha<\infty$ and $0<\delta<1$. We consider the weight $w(z)=|z|^{(\alpha+2)(p-1)(1-\delta)}$. The reader can easily convince himself that $w\in B_{p, \alpha}$. Indeed, the interesting case is to check the Carleson squares for intervals centered at $0$, or more easily, semidiscs centered at $0$, which yields $B_{p, \alpha}(w)= C_{\alpha, p}\delta^{1-p}$. Let us now consider $f(z)=|z|^{(\alpha+2)(\delta-1)}1_{\{z\in \mathcal H \mid |z|\leq 1\}}(z)$, we see that $f\in L^{p}(wdA_{\alpha})$ and $\norm f. L^{p}(wdA_{\alpha}).^{p}=C_{\alpha}\frac{1}{\delta}$.
For fixed $\alpha$, there exists a constant $M_{\alpha}>0$ such that for any $z\in \mathcal H$, $|z|\geq M_{\alpha}$,  $\arg ((z-\bar{\xi_{1}})^{2+\alpha}, (z-\bar{\xi_{2}})^{2+\alpha} )\leq \pi/2$ for every $\xi_{1},\xi_{2}\in \mathcal H$ such that $|\xi_{1}|, |\xi_{2}|\leq 1$. Simple geometric considerations (see figure \ref{f.angle}) show that for example $M_{\alpha}= \left(\tan(\frac{\pi}{4}(2+\alpha)^{-1})\right)+1$ will do.

\begin{figure}[H]
\begin{tikzpicture}  
\draw  (-5,0) -- (5,0); 
\draw (1,0) arc [radius=1, start angle=0, end angle=180];
\draw (4,0) arc [radius=4, start angle=0, end angle=180];
\draw (-1,0) -- (0,4);
\draw (1,0) -- (0,4); 
\draw (-0.1,3.5) arc [radius=0.7, start angle=260, end angle=280];
\node at (0,4.4) {(0,$M_{\alpha}$)};
\node at (0,3) {$\theta_{\alpha}$};
\node at (1.4, -0.5) {$|z|=1$};
\node at (4.4, -0.5) {$|z|=M_{\alpha}$};

\end{tikzpicture} 
\caption{Angle of the cone}
\label{f.angle}
\end{figure}

 
Therefore if $f$ is supported on $\{z\in \mathcal H \mid |z|\leq 1\}$ and nonnegative, we have for $z\in \mathcal H$, $|z|\geq M_{\alpha}$:
\begin{equation}
\label{e.lowerb}
|P_{\alpha}(f)(z)|\geq C_{\alpha} \int_{\{\xi\in \mathcal H \mid |\xi|\leq 1\}}\frac{f(\xi)}{|z-\bar{\xi}|^{2+\alpha}}dA_{\alpha}(\xi).
\end{equation}
Moreover $|z-\bar{\xi}|^{2+\alpha}\leq 2^{2+\alpha}|z|^{2+\alpha}$. This estimate together with \eqref{e.lowerb} applied to our particular function $f(z)=|z|^{(\alpha+2)(\delta-1)}1_{\{z\in \mathcal H \mid |z|\leq 1\}}(z)$ leads us to 

\begin{equation}
\label{e.sharpex}
P_{\alpha}(f)(z)\geq C'_{\alpha} |z|^{-(2+\alpha)}\int_{\{\xi\in \mathcal H \mid |\xi|\leq 1\}}|\xi|^{(\alpha+2)(\delta-1)} dA_{\alpha}(\xi)\geq C_{\theta_{0},\alpha}  \frac{1}{\delta}|z|^{-(2+\alpha)}  .
\end{equation}

We use \eqref{e.sharpex} to get the desired estimate
$$
\norm P_{\alpha}(f).L^{p}(wdA_{\alpha}).^{p}\geq \tilde{C}_{\alpha}\delta^{-(p+1)}\approx C_{\alpha,p}B_{p, \alpha}(w)^{\frac{p}{p-1}}\norm f.L^{p}(wdA_{\alpha}).^{p}.
$$



\end{document}